\documentclass[11pt]{amsart}
\usepackage{amsmath}
\usepackage{amsfonts}
\usepackage{amsthm}
\usepackage{amssymb}
\usepackage{tikz-cd}
\usepackage{graphics}
\usepackage{array}
\usepackage{enumerate}
\usepackage{dsfont}
\usepackage{color}
\usepackage{wasysym}
\usepackage{hyperref}
\usepackage{pdfsync}

\newcommand{\bel}[1]{\begin{equation}\label{#1}}

\newcommand{\be}{\begin{equation}}

\newcommand{\ba}{\begin{eqnarray}}
\newcommand{\ea}{\end{eqnarray}}

\newcommand{\qe}{\end{equation}}

\newcommand{\N}{{\mathbb N}}

\newcommand{\C}{{\mathbb C}}
\newcommand{\Q}{{\mathbb Q}}

\DeclareMathOperator{\Pic}{Pic}

\DeclareMathOperator{\GV}{GV}
\DeclareMathOperator{\IT}{IT_0}

\newcommand{\Hmm}[1]{\leavevmode{\marginpar{\tiny%
$\hbox to 0mm{\hspace*{-0.5mm}$\leftarrow$\hss}%
\vcenter{\vrule depth 0.1mm height 0.1mm width \the\marginparwidth}%
\hbox to
0mm{\hss$\rightarrow$\hspace*{-0.5mm}}$\\\relax\raggedright #1}}}

\theoremstyle{plain}

\newtheorem{thm}{Theorem}[section]

\newtheorem{prop}[thm]{Proposition}
\newtheorem{coro}[thm]{Corollary}
\newtheorem{lemma}[thm]{Lemma}

\theoremstyle{definition}

\newtheorem{defi}[thm]{Definition}
\newtheorem{exa}[thm]{Example}

\newtheorem*{ac}{Acknowledgements}

\begin{document}

\title[On vanishing and torsion-freeness results for adjoint pairs]{On vanishing and torsion-freeness results for adjoint pairs}
	
	\author{Fanjun Meng}
	\address{Department of Mathematics, Northwestern University, 2033 Sheridan Road, Evanston, IL 60208, USA}
	\email{fanjunmeng2022@u.northwestern.edu}


\begin{abstract}
We prove some vanishing and torsion-freeness results for higher direct images of adjoint pairs satisfying relative abundance and nefness conditions. These are applied to generic vanishing and weak positivity.
\end{abstract}

\maketitle
	\setcounter{tocdepth}{1}
	\tableofcontents

\section{Introduction}\label{1}

In this paper, we prove some vanishing and torsion-freeness results in the style of \cite{Kol86}, for higher direct images of adjoint pairs satisfying relative abundance and nefness conditions and give some examples and applications. We work over $\C$.

One of our original motivations was the question of whether $f_*\mathcal{O}_X(mK_X+L)$ is a GV-sheaf, where $f$ is a morphism from a smooth projective variety $X$ to an abelian variety $A$, $m$ is a positive integer and $L$ is a nef Cartier divisor. It is known that $f_*\mathcal{O}_X(mK_X)$ is a GV-sheaf by \cite{GL87, Hac04, PS14} in increasing generality, and $\mathcal{O}_X(K_X+L)$ satisfies a certain GV-property by \cite{PP11a}. However, it is unclear that whether $f_*\mathcal{O}_X(mK_X+L)$ is a GV-sheaf for arbitrary $L$. We discover that assuming that $L$ is $f$-abundant makes it true by proving some Koll\'ar-type vanishing results in this case. 

Vanishing and torsion-freeness theorems are proved for canonical bundles in \cite{Kol86}, for klt pairs in \cite{EV92, Kol95} and for log canonical pairs in \cite{Amb03, Fuj09}. Similar type results are proved for adjoint pairs satisfying abundance and nefness conditions in \cite{EV92} and twisted multiplier ideals in \cite{EP08} and obtained in the analytic setting in \cite{FM16, Mat16}. 

We prove a similar type vanishing result in the following theorem assuming only relative abundance instead of global abundance. Let $\frak{a}$ be a nonzero ideal sheaf and $c>0$ a rational number. If $L\otimes\frak{a}^c$ is nef and abundant, this theorem is known by \cite[Corollary 6.17]{EV92} and \cite[Theorem 3.2]{EP08}. For the definition of $L\otimes\frak{a}^c$ being $f$-abundant, $f$-nef or nef, see Definition \ref{ideal}.

\begin{thm}\label{main1}
Let $f$ be a surjective morphism from a klt pair $(X, \Delta)$ to a normal projective variety $Y$, $L$ a $\Q$-Cartier $\Q$-divisor on $X$ and $\frak{a}$ a nonzero ideal sheaf such that $L\otimes\frak{a}^c$ is nef and $f$-abundant, where $c>0$ is a rational number. Let $D$ be a Cartier divisor on $X$ such that $D\equiv K_X+\Delta+L+f^*N$, where $N$ is a nef and big $\Q$-Cartier $\Q$-divisor on $Y$. For $i>0$ and $j\geq 0$, we have  
	$$H^i\big(Y, R^jf_*\big(\mathcal{O}_X(D)\otimes\mathcal{J}\big((X, \Delta); \frak{a}^c\big)\big)\big)=0.$$
\end{thm}

The sheaf $\mathcal{J}\big((X, \Delta); \frak{a}^c\big)$ is the multiplier ideal associated to the nonzero ideal sheaf $\frak{a}$ and the rational number $c>0$ on the klt pair $(X, \Delta)$. There are situations where $L\otimes\frak{a}^c$ is $f$-abundant but not abundant. One example is given in Section \ref{4}, see Example \ref{CJ} for details. Note that Theorem \ref{main1} is not true if we only assume $L\otimes\frak{a}^c$ is $f$-nef and $f$-abundant (cf. Example \ref{cou}). However, the torsion-freeness result still holds if we only assume relative nefness, which is the next theorem. If $M\otimes\frak{a}^c$ is nef and abundant, this theorem is known by \cite[Theorem 3.2]{EP08}.

\begin{thm}\label{main2}
Let $f$ be a surjective morphism from a klt pair $(X, \Delta)$ to a normal projective variety $Y$, $M$ a $\Q$-Cartier $\Q$-divisor on $X$ and $\frak{a}$ a nonzero ideal sheaf such that $M\otimes \frak{a}^c$ is $f$-nef and $f$-abundant, where $c>0$ is a rational number. Let $D$ be a Cartier divisor on $X$ such that $D\equiv K_X+\Delta+M$. Then 
$$R^jf_*\big(\mathcal{O}_X(D)\otimes\mathcal{J}\big((X, \Delta); \frak{a}^c\big)\big)$$
is torsion-free for $j\geq 0$.
\end{thm}

Our proofs of the two theorems above rely on the characterization of relatively abundant divisors in \cite{KMM87} and \cite{Nak86}, see Proposition \ref{abun} for details. It also relies on the use of the vanishing and torsion-freeness theorems for klt pairs in \cite{Kol95}.

Based on Theorem \ref{main1}, we deduce a Koll\'ar-type vanishing result for the pluricanonical case following the strategy used in \cite[Theorem 1.7]{PS14}. When $L$ is nef and $f$-big, this theorem is known by \cite[Variant 1.5]{PS14}.

\begin{thm}\label{pluri}
Let $f$ be a morphism from a klt pair $(X, \Delta)$ to a projective variety $Y$ of dimension $n$, $L$ a nef and $f$-abundant $\Q$-Cartier $\Q$-divisor on $X$ and $H$ a Cartier divisor on $Y$ such that $\mathcal{O}_Y(H)$ is ample and globally generated. Let $D$ be a Cartier divisor on $X$ such that $D\equiv m(K_X+\Delta+L)+lf^*H$, where $m\geq1$ and $l>(m-1)(n+1)$ are rational numbers. Then 
$$H^i(Y, f_*\mathcal{O}_X(D))=0$$
for $i>0$. If $l>(m-1)(n+1)+n$, $f_*\mathcal{O}_X(D)$ is $0$-regular with respect to $\mathcal{O}_Y(H)$ and thus globally generated.
\end{thm}

In Section \ref{4}, we discuss several applications based on our technical theorems. We give the adjoint pair analogue of Viehweg's result on weak positivity, see Corollary \ref{wp}. In a different direction, we consider the GV-property of pushforwards of adjoint pairs satisfying nefness and relative abundance conditions under morphisms to abelian varieties, addressing the motivational question at the beginning of the introduction. This theorem is known when $m=1$ and $L=0$ by \cite{GL87, Hac04} and when $m\geq1$ and $L$ is nef and $f$-big by \cite[Variant 5.6]{PS14}.

\begin{thm}\label{GV}
Let $f$ be a morphism from a klt pair $(X, \Delta)$ to an abelian variety $A$, $L$ a nef and $f$-abundant $\Q$-Cartier $\Q$-divisor on $X$ and $D$ a Cartier divisor on $X$ such that $D\equiv m(K_X+\Delta+L)$, where $m\geq1$ is a rational number. Then $f_*\mathcal{O}_X(D)$ is a $\GV$-sheaf.
\end{thm}

It is known that pushforwards of pluricanonical bundles under morphisms to abelian varieties satisfy an even stronger property: namely they have the Chen--Jiang decomposition by \cite{CJ18, PPS17, LPS20} in increasing generality, while $f_*\mathcal{O}_X(D)$ where $D\sim_{\Q}m(K_X+\Delta)$ has the Chen--Jiang decomposition, as proved independently in \cite{Jia21} and \cite{Men21}. It is therefore natural to ask whether $f_*\mathcal{O}_X(D)$ has the Chen--Jiang decomposition under the hypotheses of Theorem \ref{GV}. This turns out to be false; we give a counterexample in Example \ref{CJ} in which $f_*\mathcal{O}_X(D)$ does not have a Chen--Jiang decomposition for any rational number $m\geq1$ and any Cartier divisor $D\equiv m(K_X+\Delta+L)$. To ensure that there exists a Cartier divisor $D$ which is numerically equivalent to $m(K_X+\Delta+L)$ such that $f_*\mathcal{O}_X(D)$ has the Chen--Jiang decomposition, it seems that we may additionally need to assume that $K_X+\Delta$ is pseudo-effective.

\begin{ac}
{I would like to express my sincere gratitude to my advisor Mihnea Popa for proposing a related problem and for helpful discussions and constant support. I thank the referee for helpful comments.}
\end{ac}

\section{Preliminaries}\label{2}

We work over $\C$ and all varieties are projective throughout the paper. A \emph{fibration} is a projective surjective morphism with connected fibers. For the definitions and basic results on the singularities of pairs we refer to \cite{KM98}. We always ask the boundary $\Delta$ in a pair $(X,\Delta)$ to be effective. For the definitions and basic results on multiplier ideals and asymptotic multiplier ideals we refer to \cite[Chapters 9 and 11]{Laz04b}.

First, we give the definition for the generic fiber of a morphism which is not necessarily surjective.

\begin{defi}\label{generic fiber}
Let $f$ be a morphism from a normal projective variety $X$ to a projective variety $Y$. The Stein factorization of $f$ gives a decomposition of $f$ as $g\circ h$, where $h$ is a fibration and $g$ is a finite morphism. The \emph{generic fiber} of $f$ is defined as the generic fiber of $h$.
\end{defi}	

Let $X$ be a normal projective variety, $L$ a $\Q$-Cartier $\Q$-divisor on $X$ and $f$ a morphism from $X$ to a projective variety $Y$. We denote the \emph{numerical equivalence} by the symbol $\equiv$. We denote the \emph{Kodaira dimension} of $L$ by $\kappa(X, L)$ and the \emph{numerical dimension} of $L$ by $\kappa_\sigma(X, L)$, see \cite[Chapter V]{Nak04} and \cite{Kaw85b}. The $\Q$-Cartier $\Q$-divisor $L$ is said to be \emph{$f$-nef} if $L\cdot C\geq0$ for any curve $C\subset X$ contracted by $f$. It is said to be \emph{$f$-big} if it is big after being restricted to the generic fiber of $f$. It is said to be \emph{abundant} if $\kappa(X, L)=\kappa_\sigma(X, L)\geq 0$. It is said to be \emph{$f$-abundant} if it is abundant after being restricted to the generic fiber of $f$.

We will need a result about the characterization of a relatively nef and relatively abundant divisor which is \cite[Proposition 6-1-3]{KMM87} and \cite[Lemma 6]{Nak86}.

\begin{prop}\label{abun}
Let $X$ be a normal variety with a proper morphism $f\colon X\to Y$ onto a variety $Y$ and $D$ a $f$-nef and $f$-abundant $\Q$-Cartier $\Q$-divisor on $X$. Then there exists a commutative diagram
\begin{center}
\begin{tikzcd}
			& X' \arrow[dl, "\mu" swap] \arrow[dr, "g"]  \\
			X  \arrow[dr, "f"] &&Z \arrow[dl, "h" swap]\\&Y 
\end{tikzcd}
\end{center}
which satisfies the following conditions:
\begin{enumerate}
	\item[$\mathrm{(i)}$] $\mu$, $g$ and $h$ are projective morphisms,

	\item[$\mathrm{(ii)}$] $X'$ and $Z$ are smooth varieties,
	
	\item[$\mathrm{(iii)}$] $\mu$ is a birational morphism and $g$ is a fibration,
	
	\item[$\mathrm{(iv)}$] There exists a $h$-nef and $h$-big $\Q$-Cartier $\Q$-divisor $B$ on $Z$ such that $\mu^*L\sim_{\Q}g^*B$.
\end{enumerate}
\end{prop}

\begin{defi}\label{ideal}
Let $f$ be a morphism from a normal projective variety $X$ to a projective variety $Y$, $c>0$ a rational number, $L$ a $\Q$-Cartier $\Q$-divisor on $X$ and $\frak{a}$ a nonzero ideal sheaf. We say that $L\otimes\frak{a}^c$ is \emph{nef}, \emph{$f$-nef} or \emph{$f$-abundant} if there exists a log resolution of $\frak{a}$ denoted by $\varphi\colon W\to X$ such that $\frak{a}\cdot\mathcal{O}_W=\mathcal{O}_W(-F)$ where $F$ is an effective divisor and $\varphi^*L-cF$ is nef, $(f\circ\varphi)$-nef or $(f\circ\varphi)$-abundant respectively.
\end{defi}

It is easy to see that Definition \ref{ideal} does not depend on the choice of the log resolution $\varphi$ and coincides with the usual notions of (relative) nefness and abundance for $\Q$-Cartier $\Q$-divisors when $\frak{a}=\mathcal{O}_X$. 

Next, we define the asymptotic multiplier ideals for $\Q$-Cartier $\Q$-divisors.

\begin{defi}\label{MFQ}
Let $f$ be a surjective morphism from a klt pair $(X, \Delta)$ to a projective variety $Y$ and $M$ a $\Q$-Cartier $\Q$-divisor on $X$. The \emph{asymptotic multiplier ideal} $\mathcal{J}\big((X, \Delta); f, ||M||\big)$ is defined as the usual multiplier ideal $\mathcal{J}\big((X, \Delta); f, \frac{1}{p}|pM|\big)$, where $p>0$ is an integer sufficiently big and divisible such that $pM$ is a Cartier divisor. For the definition of the latter, see \cite[Generalization 9.2.21]{Laz04b}.
\end{defi}

It is easy to see that Definition \ref{MFQ} does not depend on the choice of the sufficiently big and divisible integer $p>0$ and coincides with the definition of the usual asymptotic multiplier ideals when $M$ is a Cartier divisor. 

\section{Vanishing and torsion-freeness results}\label{3}

We first prove Theorem \ref{main1}, as statement (i) of the following theorem. Statement (ii) is used in the proof of Theorem \ref{main2}.

\begin{thm}\label{1}
Let $f$ be a surjective morphism from a klt pair $(X, \Delta)$ to a normal projective variety $Y$, $L$ a $\Q$-Cartier $\Q$-divisor on $X$ and $\frak{a}$ a nonzero ideal sheaf such that $L\otimes \frak{a}^c$ is nef and $f$-abundant, where $c>0$ is a rational number. Let $M$ be a $\Q$-Cartier $\Q$-divisor on $X$ and $D$ a Cartier divisor on $X$ such that $D\equiv K_X+\Delta+L+M$. Then:
\begin{enumerate}
	\item[$\mathrm{(i)}$] Assume that $M\equiv f^*N$, where $N$ is a nef and big $\Q$-Cartier $\Q$-divisor on $Y$. For $i>0$ and $j\geq 0$, we have  
	$$H^i\big(Y, R^jf_*\big(\mathcal{O}_X(D)\otimes\mathcal{J}\big((X, \Delta); \frak{a}^c\big)\big)\big)=0.$$ 
	\item[$\mathrm{(ii)}$] Assume that $M\equiv M'$, where $M'$ is a $f$-semiample $\Q$-Cartier $\Q$-divisor on $X$. Then 
$$R^jf_*\big(\mathcal{O}_X(D)\otimes\mathcal{J}\big((X, \Delta); \frak{a}^c\big)\big)$$ 
is torsion-free for $j\geq 0$.
\end{enumerate} 
\end{thm}

\begin{proof}
We can take a log resolution of $(X, \Delta)$ and $\frak{a}$ denoted by $\varphi\colon W\to X$ such that $\frak{a}\cdot\mathcal{O}_W=\mathcal{O}_W(-F)$ and
$$K_W\sim_\Q\varphi^*(K_{X}+\Delta)+E_W,$$
where the $\Q$-divisor $E_W+qF$ has simple normal crossings support for every $q\in\Q$. Since $L\otimes \frak{a}^c$ is nef and $f$-abundant, the $\Q$-divisor $\varphi^*L-cF$ is nef and $(f\circ\varphi)$-abundant by definition. We deduce
$$\varphi^*D+\lceil E_W-cF\rceil\equiv K_W-E_W+\lceil E_W-cF\rceil+\varphi^*L+\varphi^*M$$
$$\equiv K_W+\lceil E_W-cF\rceil-(E_W-cF)+\varphi^*L-cF+\varphi^*M.$$
The pair $(W, \lceil E_W-cF\rceil-(E_W-cF))$ is klt since $E_W-cF$ has simple normal crossings support. By \cite[Theorem 10.19]{Kol95}, we deduce that 
$$R^j\varphi_*\mathcal{O}_W(\varphi^*D+\lceil E_W-cF\rceil)=0$$
for $j>0$ since $\varphi^*L-cF+\varphi^*M$ is $\varphi$-nef and $\varphi$ is birational. Then we deduce that for $j\geq0$ 
$$R^j(f\circ\varphi)_*\mathcal{O}_W(\varphi^*D+\lceil E_W-cF\rceil)\cong R^jf_*(\varphi_*\mathcal{O}_W(\varphi^*D+\lceil E_W-cF\rceil))$$
$$\cong R^jf_*\big(\mathcal{O}_X(D)\otimes\mathcal{J}\big((X, \Delta); \frak{a}^c\big)\big)$$
by Grothendieck spectral sequence and the definition of $\mathcal{J}\big((X, \Delta); \frak{a}^c\big)$. Thus we can assume $\frak{a}=\mathcal{O}_X$ and $L$ is nef and $f$-abundant from the start. 

We prove statement (i) first. Since $L$ is $f$-nef and $f$-abundant, we have the following commutative diagram by Proposition \ref{abun}.
\begin{center}
\begin{tikzcd}
			& X' \arrow[dl, "\mu" swap] \arrow[dr, "g"]  \\
			X  \arrow[dr, "f"] &&Z \arrow[dl, "h" swap]\\&Y 
\end{tikzcd}
\end{center}
The morphisms $\mu$, $g$ and $h$ are projective and thus the varieties $X'$ and $Z$ are projective. The varieties $X'$ and $Z$ are smooth, $\mu$ is a birational morphism, $g$ is a fibration and there exists a $h$-nef and $h$-big $\Q$-divisor $B$ on $Z$ such that $\mu^*L\sim_{\Q}g^*B$. We can choose $\mu$ such that it is a log resolution of $(X, \Delta)$. Then we have
$$K_{X'}+\Delta_{X'}\sim_\Q\mu^*(K_{X}+\Delta)+E,$$
where the $\Q$-divisors $\Delta_{X'}$ and $E$ are effective and have no common components and $E$ is $\mu$-exceptional. Since $B$ is $h$-nef and $h$-big and $N$ is nef and big, we have that
$$B\sim_{\Q} E_B+H_B \quad \text{and} \quad N\sim_{\Q} E_N+H_N,$$
where $H_B$, $H_N$, $E_B\geq 0$ and $E_N\geq 0$ are $\Q$-Cartier $\Q$-divisors, $H_B$ is $h$-ample and $H_N$ is ample. We deduce
$$\mu^*D+\lceil E\rceil\equiv K_{X'}+\Delta_{X'}+\lceil E\rceil-E+\mu^*L+\mu^*M$$
$$\equiv K_{X'}+\Delta_{X'}+\lceil E\rceil-E+g^*B+g^*h^*N.$$
We define an effective divisor $T$ as $\Delta_{X'}+\lceil E\rceil-E$ and thus $(X', T)$ is klt by the definition of $\Delta_{X'}$ and $E$. By the same argument as above, we have that
$$R^j(h\circ g)_*\mathcal{O}_{X'}(\mu^*D+\lceil E\rceil)\cong R^j(f\circ \mu)_*\mathcal{O}_{X'}(\mu^*D+\lceil E\rceil)$$
$$\cong R^jf_*\mu_*\mathcal{O}_{X'}(\mu^*D+\lceil E\rceil)\cong R^jf_*\mathcal{O}_{X}(D)$$
for every $j\in \N$. We can choose two rational numbers $\varepsilon>0$ and $\delta>0$ which are sufficiently small such that $(X', T':=T+\varepsilon\delta g^*E_B+\delta g^*h^*E_N)$ is klt and $\varepsilon H_B+h^*H_N$ is ample since $H_B$ is $h$-ample and $H_N$ is ample. We deduce that 
$$\mu^*D+\lceil E\rceil-K_{X'}\equiv T+(1-\varepsilon\delta)g^*B+\varepsilon\delta g^*B+\delta g^*h^*N+(1-\delta)g^*h^*N$$
$$\equiv T+(1-\varepsilon\delta)g^*B+\varepsilon\delta g^*(E_B+H_B)+\delta g^*h^*(E_N+H_N)+(1-\delta)g^*h^*N$$
$$\equiv T'+g^*((1-\varepsilon\delta)B+\delta(\varepsilon H_B+h^*H_N))+(1-\delta)g^*h^*N.$$
Since $L$ is nef, we deduce that $B$ is nef and $(1-\varepsilon\delta)B+\delta(\varepsilon H_B+h^*H_N)$ is ample. Thus we can choose an effective $\Q$-divisor $E'\sim_{\Q}g^*((1-\varepsilon\delta)B+\delta(\varepsilon H_B+h^*H_N))$ such that the pair $(X', T'+E')$ is still klt. Since $\mu^*D+\lceil E\rceil\equiv K_{X'}+T'+E'+(1-\delta)g^*h^*N$, we deduce that
$$H^i(Y, R^jf_*\mathcal{O}_{X}(D))\cong H^i(Y, R^j(h\circ g)_*\mathcal{O}_{X'}(\mu^*D+\lceil E\rceil))=0$$
for $i>0$ and $j\geq 0$ by \cite[Theorem 10.19]{Kol95}.

Next, we prove statement (ii). Note that $R^jf_*\mathcal{O}_{X}(D)$ is torsion-free for $j\geq 0$ under the hypotheses of statement (i) by the proof as above. Since $M'$ is $f$-semiample, we can choose an integer $N>0$ sufficiently big and divisible such that the natural morphism $f^*f_*\mathcal{O}_X(NM')\to\mathcal{O}_X(NM')$ is surjective. Then we can take a Cartier divisor $H$ on $Y$ which is sufficiently ample  such that $f_*\mathcal{O}_X(NM')\otimes\mathcal{O}_Y(H)$ is globally generated. Thus $\mathcal{O}_X(NM'+f^*H)$ is globally generated and we can choose an effective $\Q$-divisor $E\sim_{\Q}NM'+f^*H$ such that $(X, \Delta+\frac{E}{N})$ is klt. Then we have
$$D+f^*H\equiv K_X+\Delta+L+M'+\frac{1}{N}f^*H+(1-\frac{1}{N})f^*H$$
$$\equiv K_X+\Delta+\frac{E}{N}+L+(1-\frac{1}{N})f^*H.$$
The proof of statement (i) implies that $R^jf_*\mathcal{O}_{X}(D)\otimes\mathcal{O}_Y(H)$ is torsion-free for $j\geq 0$ and thus $R^jf_*\mathcal{O}_{X}(D)$ is torsion-free for $j\geq 0$ since $\mathcal{O}_Y(H)$ is locally free.
\end{proof}

The following simple example shows that Theorem \ref{main1} does not hold if we only assume that $L\otimes \frak{a}^c$ is $f$-nef and $f$-abundant.

\begin{exa}\label{cou}
Let $Y$ be an elliptic curve and $H$ an ample Cartier divisor on $Y$. Define $X:=\mathbf{P}(\mathcal{O}_Y\oplus\mathcal{O}_Y(-H))$ and let $f:X\to Y$ be the projection. We consider the line bundle $\mathcal{O}_X(1)$ and fix the unique section $C$ of $f$ such that $\mathcal{O}_X(C)\cong\mathcal{O}_X(1)$. Then $3C$ is $f$-ample but it is not nef since $C^2=-\deg H<0$. We have that $K_X\sim -2C-f^*H$ by \cite[Lemma 2.10]{Har77}. We define a Cartier divisor $D:=K_X+3C+f^*H\sim C$ and 
$$H^1(Y, f_*\mathcal{O}_X(D))\cong H^1(Y, \mathcal{O}_Y\oplus\mathcal{O}_Y(-H))\neq 0$$ 
since $Y$ is an elliptic curve.
\end{exa}

Next, we prove two technical results which are used in the proof of Theorem \ref{main2}.

\begin{lemma}\label{tf}
Let $f$ be a surjective morphism from a klt pair $(X, \Delta)$ to a normal projective variety $Y$, $L$ a nef and $f$-abundant $\Q$-Cartier $\Q$-divisor on $X$ and $M$ a $\Q$-Cartier $\Q$-divisor on $X$. Let $D$ be a Cartier divisor on $X$ such that $D\equiv K_X+\Delta+L+M$. Then 
$$R^jf_*\big(\mathcal{O}_X(D)\otimes\mathcal{J}\big((X, \Delta); f, ||M||\big)\big)$$
is torsion-free for $j\geq 0$.
\end{lemma}

\begin{proof}
If $\mathcal{J}\big((X, \Delta); f, ||M||\big)=0$, the conclusion is trivial, so we assume that $\mathcal{J}\big((X, \Delta); f, ||M||\big)\neq0$ from the start. We can choose an integer $p$ sufficiently big and divisible such that $pM$ is Cartier and $\mathcal{J}\big((X, \Delta); f, ||M||\big)=\mathcal{J}\big((X, \Delta); f, \frac{1}{p}|pM|\big)$. We consider the following adjoint morphism 
$$f^*f_*\mathcal{O}_{X}(pM)\to \mathcal{O}_{X}(pM).$$
The image of this morphism is $\mathcal{O}_{X}(pM)\otimes\mathcal{I}_p$, where $\mathcal{I}_p$ is the relative base ideal of the linear series $|pM|$. We can take a log resolution of $(X, \Delta)$ and $\mathcal{I}_p$ denoted by $\varphi\colon W\to X$ such that $\mathcal{I}_p\cdot\mathcal{O}_W=\mathcal{O}_W(-F_p)$ and
$$K_W\sim_\Q\varphi^*(K_{X}+\Delta)+E_W,$$
where the $\Q$-divisor $E_W+qF_p$ has simple normal crossings support for every $q\in\Q$. The image of the new adjoint morphism
$$\varphi^*f^*f_*\mathcal{O}_{X}(pM)\to \mathcal{O}_{W}(\varphi^*(pM))$$
is $\mathcal{O}_{W}(\varphi^*(pM)-F_p)$. We can take a Cartier divisor $H$ which is sufficiently ample such that $f_*\mathcal{O}_{X}(pM)\otimes\mathcal{O}_Y(H)$ is globally generated. Then we deduce that $\mathcal{O}_{W}(\varphi^*(pM)-F_p+\varphi^*f^*H)$ is globally generated and thus $\varphi^*M-\frac{1}{p}F_p$ is $(f\circ\varphi)$-semiample. We have
$$\varphi^*D+\lceil E_W-\frac{1}{p}F_p\rceil\equiv K_W-E_W+\lceil E_W-\frac{1}{p}F_p\rceil+\varphi^*L+\varphi^*M$$
$$\equiv K_W+\lceil E_W-\frac{1}{p}F_p\rceil-(E_W-\frac{1}{p}F_p)+\varphi^*L+\varphi^*M-\frac{1}{p}F_p.$$
The pair $(W, \lceil E_W-\frac{1}{p}F_p\rceil-(E_W-\frac{1}{p}F_p))$ is klt since $E_W-\frac{1}{p}F_p$ has simple normal crossings support. Since $\varphi^*L$ is nef and $(f\circ\varphi)$-abundant and $\varphi^*M-\frac{1}{p}F_p$ is $(f\circ\varphi)$-semiample, we deduce that $R^j(f\circ\varphi)_*\mathcal{O}_W(\varphi^*D+\lceil E_W-\frac{1}{p}F_p\rceil)$ is torsion-free for $j\geq0$ by Theorem \ref{1}. Since $\varphi^*L+\varphi^*M-\frac{1}{p}F_p$ is $\varphi$-nef and $\varphi$ is birational, we deduce that 
$$R^j(f\circ\varphi)_*\mathcal{O}_W(\varphi^*D+\lceil E_W-\frac{1}{p}F_p\rceil)\cong R^jf_*\big(\mathcal{O}_X(D)\otimes\mathcal{J}\big((X, \Delta); f, ||M||\big)\big)$$
for $j\geq0$ by the same argument as in Theorem \ref{1} and this finishes the proof.
\end{proof}

\begin{prop}\label{tmi}
Let $f$ be a surjective morphism from a klt pair $(X, \Delta)$ to a projective variety $Y$ and $M$ a $f$-nef and $f$-abundant $\Q$-Cartier $\Q$-divisor on $X$. Then $\mathcal{J}\big((X, \Delta); f, ||M||\big)=\mathcal{O}_X$.
\end{prop}

\begin{proof}
By Proposition \ref{abun} and the proof of \cite[Lemma 5.11]{EV92}, we deduce that there exist a birational morphism $\mu\colon W\to X$ where $W$ is smooth and projective and an effective divisor $F$ on $W$ such that $\mu^*M-\varepsilon F$ is $(f\circ\mu)$-semiample for any rational number $\varepsilon>0$ sufficiently small. We can choose $\mu$ such that it is also a log resolution of $(X, \Delta)$. Then we have
$$K_W+\Delta_W\sim_\Q\mu^*(K_{X}+\Delta)+E,$$
where the $\Q$-divisors $\Delta_W$ and $E$ are effective and have no common components and $E$ is $\mu$-exceptional. We can choose a rational number $\varepsilon>0$ sufficiently small such that $\mu^*M-\varepsilon F$ is $(f\circ\mu)$-semiample and $(W, \Delta_W+\varepsilon F)$ is klt since $(W, \Delta_W)$ is klt. We fix this $\varepsilon$. We can choose an integer $p$ sufficiently big and divisible such that $p(\mu^*M-\varepsilon F)$ and $p\varepsilon F$ are Cartier, $\mathcal{J}\big((X, \Delta); f, ||M||\big)=\mathcal{J}\big((X, \Delta); f, \frac{1}{p}|pM|\big)$ and the following adjoint morphism 
$$(f\circ\mu)^*(f\circ\mu)_*\mathcal{O}_W(p(\mu^*M-\varepsilon F))\to\mathcal{O}_W(p(\mu^*M-\varepsilon F))$$
is surjective. We consider the following adjoint morphism 
$$(f\circ\mu)^*f_*\mathcal{O}_{X}(pM)\to\mathcal{O}_{W}(\mu^*(pM)).$$
The image of this morphism is $\mathcal{O}_{W}(\mu^*(pM))\otimes\mathcal{I}_p$, where $\mathcal{I}_p$ is the relative base ideal of the linear series $|\mu^*(pM)|$. We can take a log resolution of $(W, \Delta_W+\varepsilon F)$ and $\mathcal{I}_p$ denoted by $\nu\colon V\to W$ such that $\mu\circ\nu$ is a log resolution of $(X, \Delta)$ and the relative base ideal of the linear series $|pM|$. We have that
$\mathcal{I}_p\cdot\mathcal{O}_V=\mathcal{O}_V(-F_p)$ and
$$K_V\sim_\Q\nu^*(K_{W}+\Delta_W+\varepsilon F)+E'.$$
Then we deduce that
$$K_V\sim_\Q\nu^*(K_{W}+\Delta_W)+\nu^*(\varepsilon F)+E'$$
$$\sim_\Q\nu^*\mu^*(K_{X}+\Delta)+\nu^*E+\nu^*(\varepsilon F)+E'.$$
We consider the following commutative diagram.
	\begin{center}
	\begin{tikzcd}
			(f\circ\mu\circ\nu)^*(f\circ\mu)_*\mathcal{O}_W(p(\mu^*M-\varepsilon F))\arrow[r, "a"] \arrow[d] & \nu^*\mathcal{O}_W(p(\mu^*M-\varepsilon F)) \arrow[d] \\
			 (f\circ\mu\circ\nu)^*f_*\mathcal{O}_{X}(pM)  \arrow[r, "b"] & \mathcal{O}_{V}(\nu^*\mu^*(pM))
	\end{tikzcd}
	\end{center}
Since $a$ is surjective and the image of $b$ is $\mathcal{O}_{V}(\nu^*\mu^*(pM)-F_p)$, we deduce that $\mathcal{O}_{V}(\nu^*\mu^*(pM)-F_p)$ contains $\mathcal{O}_V(\nu^*\mu^*(pM)-\nu^*(p\varepsilon F))$ and thus $\nu^*(p\varepsilon F)\geq F_p$. By our choices of $\nu$ and $\mu$, we have that
$$\mathcal{J}\big((X, \Delta); f, ||M||\big)=\mu_*\nu_*\mathcal{O}_V(\lceil \nu^*E+\nu^*(\varepsilon F)+E'-\frac{F_p}{p}\rceil).$$
Since $E$ is effective and $\nu^*(p\varepsilon F)\geq F_p$, we deduce that 
$$\lceil\nu^*E+\nu^*(\varepsilon F)+E'-\frac{F_p}{p}\rceil\geq \lceil\nu^*E+E'\rceil\geq\lceil E'\rceil.$$
Since $(W, \Delta_W+\varepsilon F)$ is klt, we deduce that $\nu_*\mathcal{O}_V(\lceil E'\rceil)=\mathcal{O}_W$ and thus $\mathcal{J}\big((X, \Delta); f, ||M||\big)=\mathcal{O}_X$.
\end{proof}

\begin{proof}[Proof of Theorem \ref{main2}]
By the same argument as in Theorem \ref{1}, we can assume that $\frak{a}=\mathcal{O}_X$ and $M$ is $f$-nef and $f$-abundant. Then the conclusion is a direct corollary of Lemma \ref{tf} and Proposition \ref{tmi} by choosing $L=0$ in Lemma \ref{tf}.
\end{proof}

Next, we prove Theorem \ref{pluri} based on Theorem \ref{main1}. It is used to give several applications in Section \ref{4}.

\begin{proof}[Proof of Theorem \ref{pluri}]
If $m=1$, this is a direct corollary of Theorem \ref{main1} by taking a Stein factorization of $f$. 

In general, we can take a log resolution of $(X, \Delta)$ and the relative base ideal of the linear series $|D|$ denoted by $\mu$ such that 
$$K_W+\Delta_W\sim_\Q\mu^*(K_{X}+\Delta)+E$$
and the image of $\mu^*f^*f_*\mathcal{O}_X(D)\to\mu^*\mathcal{O}_X(D)$ is a line bundle denoted by $\mathcal{O}_W(\mu^*D-F)$, where the $\Q$-divisors $\Delta_W$ and $E$ are effective and have no common components, $E$ is $\mu$-exceptional, $F$ is an effective Cartier divisor and $\Delta_W+E+F$ has simple normal crossings support. Then we have
$$\mu^*D+\lceil mE\rceil\equiv m\mu^*(K_X+\Delta+L)+\lceil mE\rceil+l\mu^*f^*H$$
$$\equiv m(K_W+\Delta_W+\frac{1}{m}(\lceil mE\rceil-mE)+\mu^*L)+l\mu^*f^*H.$$
Since $m\geq1$, we deduce that $(W, \Delta_W+\frac{1}{m}(\lceil mE\rceil-mE))$ is klt. We have $\mu^*L$ is nef and $(f\circ\mu)$-abundant since $L$ is nef and $f$-abundant. Since $\lceil mE\rceil$ is $\mu$-exceptional, we have $\mu_*\mathcal{O}_W(\mu^*D+\lceil mE\rceil)\cong\mathcal{O}_X(D)$. Thus we can assume from the start that $(X, \Delta)$ is log smooth and the image of the adjoint morphism $f^*f_*\mathcal{O}_X(D)\to\mathcal{O}_X(D)$ is a line bundle denoted by $\mathcal{O}_X(D-F)$, where $F$ is an effective Cartier divisor such that $\Delta+F$ has simple normal crossings support.

Since $H$ is ample, there exists a smallest integer $k\geq -l$ such that $f_*\mathcal{O}_X(D)\otimes\mathcal{O}_Y(kH)$ is globally generated. Since $f^*f_*\mathcal{O}_X(D)\to\mathcal{O}_X(D-F)$ is surjective, we deduce that $\mathcal{O}_X(D+kf^*H-F)$ is globally generated. Thus we can choose a smooth divisor $B\sim D+kf^*H-F$ such that $B$ and $\Delta+F$ have no common components and $\Delta+F+B$ has simple normal crossings support. We define an effective divisor $T$ as $\lfloor\Delta+\frac{m-1}{m}F\rfloor$. We deduce that
$$D-T\equiv K_X+\Delta+L+(m-1)(K_X+\Delta+L)+lf^*H-T$$
$$\equiv K_X+\Delta+L+\frac{m-1}{m}(D-lf^*H)+lf^*H-T$$
$$\equiv K_X+\Delta+L+\frac{m-1}{m}(B+F-(k+l)f^*H)+lf^*H-T$$
$$\equiv K_X+\Delta+\frac{m-1}{m}F-T+\frac{m-1}{m}B+L+(l-\frac{(m-1)(k+l)}{m})f^*H.$$
Since $B$ and $\Delta+F$ have no common components and $\Delta+F+B$ has simple normal crossings support, $(X, \Delta+\frac{m-1}{m}F-T+\frac{m-1}{m}B)$ is klt. Since the coefficients of $\Delta$ are smaller than 1, we deduce that $T\leq F$. We have that $f_*\mathcal{O}_{X}(D-F')\cong f_*\mathcal{O}_{X}(D)$ for any effective divisor $F'\leq F$ since we can factor the adjoint morphism as
$$f^*f_*\mathcal{O}_{X}(D)\to \mathcal{O}_{X}(D-F)\hookrightarrow \mathcal{O}_{X}(D-F')\hookrightarrow \mathcal{O}_{X}(D)$$
and the morphism induced from pushforwards
$$f_*\mathcal{O}_{X}(D)\to f_*\mathcal{O}_{X}(D-F)\hookrightarrow f_*\mathcal{O}_{X}(D-F')\hookrightarrow f_*\mathcal{O}_{X}(D)$$
is identity. Thus we have $f_*\mathcal{O}_{X}(D-T)\cong f_*\mathcal{O}_{X}(D)$. Since we have proved the theorem when $m=1$, we deduce that
$$H^i(Y, f_*\mathcal{O}_X(D)\otimes\mathcal{O}_Y(bH))\cong H^i(Y, f_*\mathcal{O}_X(D-T)\otimes\mathcal{O}_Y(bH))=0$$
for every $i>0$ and every integer $b>\frac{(m-1)(k+l)}{m}-l$. We deduce that $f_*\mathcal{O}_X(D)\otimes\mathcal{O}_Y(bH)$ is $0$-regular with respect to $\mathcal{O}_Y(H)$ and thus globally generated for every integer $b>\frac{(m-1)(k+l)}{m}+n-l\geq -l$. By the definition of $k$, we have 
$$k\leq\frac{(m-1)(k+l)}{m}+1+n-l,$$
which is the same as $k\leq m+mn-l$. We deduce
$$H^i(Y, f_*\mathcal{O}_X(D)\otimes\mathcal{O}_Y(bH))=0$$
for every $i>0$ and every integer $b>\frac{(m-1)(m+mn-l+l)}{m}-l=(m-1)(n+1)-l$. Since $l>(m-1)(n+1)$, we can choose $b=0$ and this finishes the proof.
\end{proof}

\section{Applications}\label{4}

First, we give an application towards the GV-property of sheaves on abelian varieties by proving Theorem \ref{GV}. For the notions of GV-sheaves, M-regular sheaves and sheaves satisfying $\IT$, we refer to \cite[Definition 2.4]{Men21}.

\begin{proof}[Proof of Theorem \ref{GV}]
Let $M=H^{\otimes l}$, where $l>(m-1)(\dim A+1)$ is an integer which can be chosen sufficiently big and $H$ is an ample and globally generated line bundle on $\hat{A}$ which is the dual abelian variety of $A$. Let $\varphi_M:\hat{A}\to A$ be the isogeny induced by $M$. We have the following base change diagram. 
	\begin{center}
	\begin{tikzcd}
			X' \arrow[r, "\psi"] \arrow[d, "f'"] & X \arrow[d, "f"] \\
			 \hat{A} \arrow[r, "\varphi_M" ] & A 
	\end{tikzcd}
	\end{center}
By \cite[Corollary 3.1]{Hac04}, we only need to prove that
$$H^i(\hat{A}, \varphi_M^*f_*\mathcal{O}_X(D)\otimes M)\cong H^i(\hat{A}, f'_*\mathcal{O}_{X'}(\psi^*D)\otimes H^{\otimes l})$$
vanishes for $i>0$. We define a $\Q$-divisor $\Delta'$ by $K_{X'}+\Delta'=\psi^*(K_X+\Delta)$. Since $\psi$ is an \'etale morphism, $\Delta'$ is effective and $(X', \Delta')$ is klt. Since $L$ is nef and $f$-abundant, $\psi^*L$ is nef and $f'$-abundant by the definition of relatively abundant divisors. We deduce that 
$$\psi^*D\equiv \psi^*(m(K_X+\Delta+L))\equiv m(K_{X'}+\Delta'+\psi^*L).$$
By Theorem \ref{pluri}, we deduce that
$$H^i(\hat{A}, f'_*\mathcal{O}_{X'}(\psi^*D)\otimes H^{\otimes l})=0$$
for every $i>0$ and every $l>(m-1)(\dim A+1)$ and this finishes the proof.
\end{proof}

We remark that we can also prove $R^jf_*\mathcal{O}_X(D)$ is a GV-sheaf for $j\geq0$ if $m=1$ under the hypotheses of Theorem \ref{GV} by the same method as above and Theorem \ref{main1}. Note that Theorem \ref{GV} is not true if we only assume $L$ is $f$-nef and $f$-abundant. For example, $f_*\mathcal{O}_X(D)$ in Example \ref{cou} is not a GV-sheaf since GV-sheaves are nef by \cite[Theorem 4.1]{PP11b}, while $f_*\mathcal{O}_X(D)$ is not nef.

The following corollary is a standard consequence of Theorem \ref{GV}.

\begin{coro}
Let $f$ be a morphism from a klt pair $(X, \Delta)$ to an abelian variety $A$, $L$ a nef and $f$-abundant $\Q$-Cartier $\Q$-divisor on $X$ and $D$ a Cartier divisor on $X$ such that $D\equiv m(K_X+\Delta+L)$, where $m\geq1$ is a rational number. Let $H$ be an ample line bundle on $A$. Then:
\begin{enumerate}
	\item[$\mathrm{(i)}$] $f_*\mathcal{O}_X(D)$ is a nef sheaf.

	\item[$\mathrm{(ii)}$] $H^i(A, f_*\mathcal{O}_X(D)\otimes H)=0$ for $i>0$.

	\item[$\mathrm{(iii)}$] $f_*\mathcal{O}_X(D)\otimes H^{\otimes 2}$ is globally generated.
\end{enumerate}
\end{coro}

\begin{proof}
By \cite[Theorem 4.1]{PP11b}, every GV-sheaf is nef. Then statement (i) follows from Theorem \ref{GV}. The tensor product of a locally free sheaf satisfying $\IT$ and a GV-sheaf satisfies $\IT$ by \cite[Proposition 3.1]{PP11b} and this proves statement (ii). Statement (iii) follows from \cite[Theorem 2.4]{PP03} since coherent sheaves satisfying $\IT$ are M-regular.
\end{proof}

We recall the definition of the Chen--Jiang decomposition, which is intended to model the behavior of pushforwards of canonical bundles in \cite{CJ18}.

\begin{defi}
Let $\mathcal{F}$ be a coherent sheaf on an abelian variety $A$. The sheaf $\mathcal{F}$ is said to have the \emph{Chen--Jiang decomposition} if $\mathcal{F}$ admits a finite direct sum decomposition
$$\mathcal{F}\cong \bigoplus_{i\in I}(\alpha_i\otimes p_i^*\mathcal{F}_i),$$
where each $A_i$ is an abelian variety, each $p_i\colon A\to A_i$ is a fibration, each $\mathcal{F}_i$ is a nonzero M-regular coherent sheaf on $A_i$, and each $\alpha_i\in\Pic^0(A)$ is a torsion line bundle.
\end{defi}

We have that the sheaf $f_*\mathcal{O}_X(D)$ addressed in Theorem \ref{GV} is a GV-sheaf. However, there are situations when $f_*\mathcal{O}_X(D)$ does not have a Chen--Jiang decomposition under the hypotheses of Theorem \ref{GV} for any rational number $m\geq1$ and any Cartier divisor $D\equiv m(K_X+\Delta+L)$. We will show that \cite[Example 1.1]{Sho00} satisfies this property in the following.

\begin{exa}\label{CJ}
Let $Y$ be an elliptic curve and $E$ a rank $2$ vector bundle on $Y$ which is the unique non-split extension
$$0\to\mathcal{O}_Y\to E\to\mathcal{O}_Y\to0.$$
Define $X:=\mathbf{P}(E)$ and let $f:X\to Y$ be the projection. We consider the line bundle $\mathcal{O}_X(1)$ and fix the unique section $C$ of $f$ such that $\mathcal{O}_X(C)\cong\mathcal{O}_X(1)$. By \cite[Example 1.1]{Sho00}, if there exists a curve $C'$ on $X$ such that $C'\equiv mC$ for some rational number $m$, then $m$ is a positive integer and $C'=mC$. It implies that $\kappa(X, C)=0$. We know $C$ is $f$-abundant and $C^2=0$. We deduce that $C$ is nef and $\kappa_\sigma(X, C)=1\neq\kappa(X, C)$. Thus $C$ is not abundant. 

We have that $K_X\sim -2C$ by \cite[Lemma 2.10]{Har77}. We define $L:=3C$ which is nef and $f$-abundant. If $f_*\mathcal{O}_X(D)$ has the Chen--Jiang decomposition for some rational number $m\geq1$ and some Cartier divisor $D\equiv m(K_X+L)$, we deduce that the continuous evaluation morphism
$$\mathcal{S}=\bigoplus_{\alpha\in \Pic^0(Y)\atop \mathrm{torsion}}H^0(Y, f_*\mathcal{O}_X(D)\otimes\alpha)\otimes\alpha^{-1}\to f_*\mathcal{O}_X(D)$$
is surjective by \cite[Theorem 5.1]{LPS20} and \cite[Lemma 2.9]{Men21}. Since $D\equiv m(K_X+L)\equiv mC$, we deduce that $m$ is an integer and $\mathcal{O}_X(D)\cong \mathcal{O}_X(mC)\otimes f^*P$ for some $P\in\Pic^0(Y)$ by \cite[Proposition 2.3]{Har77}. We know that $f_*\mathcal{O}_X(mC)\cong S^mE$ has a filtration whose successive quotients are $\mathcal{O}_Y$ by the definition of $E$, where $S^mE$ is the $m$-th symmetric product of $E$. We deduce that $f_*\mathcal{O}_X(mC)\otimes Q$ has a filtration whose successive quotients are $Q$ for every $Q\in\Pic^0(Y)$ and thus 
$$H^i(Y, f_*\mathcal{O}_X(mC)\otimes Q)=0$$ 
for every $i\geq0$ and every nontrivial $Q\in\Pic^0(Y)$. If 
$$H^0(Y, f_*\mathcal{O}_X(D)\otimes\alpha)\cong H^0(Y, f_*\mathcal{O}_X(mC)\otimes P\otimes\alpha)\neq0,$$
we deduce that $P$ is a torsion element and $\alpha\cong P^{-1}$. We know $$H^0(Y, f_*\mathcal{O}_X(mC))\cong\C$$
since $\kappa(X, C)=0$. Since the rank of $\mathcal{S}$ is at most $1$ and the rank of $f_*\mathcal{O}_X(D)$ is $m+1\geq2$, the continuous evaluation morphism cannot be surjective, which is a contradiction. Thus $f_*\mathcal{O}_X(D)$ does not have a Chen--Jiang decomposition for any rational number $m\geq1$ and any Cartier divisor $D\equiv m(K_X+L)$ in this example.
\end{exa}

Next, we give several applications to weak positivity based on Theorem \ref{pluri}. We omit the proofs since they follow from well-known strategies which rely on Viehweg's fiber product trick, see \cite[Theorem III]{Vie83}, \cite[Theorem 4.13]{Cam04}, \cite[Theorem 3.30]{Hor10}, \cite[Theorems 1.8 and 4.4]{PS14}, \cite[Theorem 1.1]{Fuj17} and \cite[Theorems D and E]{DM19} for details. We recall the definition of weak positivity.

\begin{defi}[cf. \cite{Vie83}]
A torsion-free coherent sheaf $\mathcal{F}$ on a projective variety $X$ is said to be \emph{weakly positive} on a nonempty open set $U\subseteq X$ if for every ample line bundle $H$ on $X$ and every $a\in\N$, the sheaf $S^{[ab]}\mathcal{F}\otimes H^{\otimes b}$ is generated by global sections at each point of $U$ for $b$ sufficiently big, where $S^{[ab]}\mathcal{F}$ is the reflexive hull of the $ab$-th symmetric product $S^{ab}\mathcal{F}$. We say $\mathcal{F}$ is \emph{weakly positive} if it is weakly positive on some open set $U$.
\end{defi}

If $X$ is smooth, $\Delta=0$ and $L$ is nef and $f$-big, the following theorem is known by \cite[Theorem 4.4]{PS14}. If $(X, \Delta)$ is log canonical and $L=0$, it is known by \cite[Theorem E]{DM19}.

\begin{thm}\label{ggg}
Let $f$ be a fibration from a klt pair $(X, \Delta)$ to a smooth projective variety $Y$ of dimension $n$, $L$ a nef and $f$-abundant $\Q$-Cartier $\Q$-divisor on $X$, $H$ an ample and globally generated line bundle on $Y$ and $A=\omega_Y\otimes H^{\otimes n+1}$. Let $D$ be a Cartier divisor on $X$ such that $D\equiv m(K_{X/Y}+\Delta+L)$, where $m\geq1$ is an integer. Then there exists a nonempty open set $U\subseteq Y$ such that
$$(f_*\mathcal{O}_X(D))^{[s]}\otimes A^{\otimes l}$$
is generated by global sections on $U$ for every $s\geq1$ and $l\geq m$, where $(f_*\mathcal{O}_X(D))^{[s]}$ is the reflexive hull of $(f_*\mathcal{O}_X(D))^{\otimes s}$.
\end{thm}

As a standard consequence, we obtain the following corollary. 

\begin{coro}\label{wp}
Let $f$ be a fibration from a klt pair $(X, \Delta)$ to a smooth projective variety $Y$, $L$ a nef and $f$-abundant $\Q$-Cartier $\Q$-divisor on $X$ and $D$ a Cartier divisor on $X$ such that $D\equiv m(K_{X/Y}+\Delta+L)$, where $m\geq1$ is an integer. Then $f_*\mathcal{O}_X(D)$ is weakly positive.
\end{coro}

It was first proved in \cite[Theorem III]{Vie83} for pushforwards of relative pluricanonical bundles. If $X$ is smooth, $\Delta=0$ and $L$ is nef and $f$-big, it is known by \cite[Theorem 3.30]{Hor10} when $m=1$ and by \cite[Theorem 1.8]{PS14} when $m\geq1$. If $(X, \Delta)$ is log canonical and $L=0$, it is known by \cite[Theorem 4.13]{Cam04}, \cite[Theorem 1.1]{Fuj17} and \cite[Theorem D]{DM19}.

	\bibliographystyle{amsalpha}
	\bibliography{biblio}	

\end{document}